\documentclass[11pt,leqno]{article}

\usepackage{amssymb,amsmath,amsthm}

\newcommand{\n}{\noindent}
\newcommand{\bb}[1]{\mathbb{#1}}
\newcommand{\cl}[1]{\mathcal{#1}}
\newcommand{\vp}{\varepsilon}
\newcommand{\ovl}{\overline}

\theoremstyle{plain}

\newtheorem{thm}{Theorem}[section]
 
\newtheorem{pro}[thm]{Proposition}
\newtheorem*{prop}{Proposition}

\newtheorem{lem}[thm]{Lemma}
\newtheorem{cor}[thm]{Corollary}

\theoremstyle{remark}
\newtheorem*{rk}{Remark}
\newtheorem{rem}[thm]{Remark}

\theoremstyle{definition}

\numberwithin{equation}{section}

\begin{document}

\title{Real interpolation between row and column spaces}

\author{by\\
Gilles Pisier\\
Texas A\&M University\\
College Station, TX 77843, U. S. A.\\
and\\
Universit\'e Paris VI\\
Equipe d'Analyse, Case 186, 75252\\
Paris Cedex 05, France}

%\date{}

\maketitle

\begin{abstract} 
We give an equivalent expression for the $K$-functional associated to the pair of operator spaces $(R,C)$ formed by the rows and columns respectively. This yields a description of the real interpolation spaces for the pair $(M_n(R), M_n(C))$ (uniformly over $n$). More generally, the same result is valid when $M_n$ (or $B(\ell_2)$) is replaced by any semi-finite von~Neumann algebra. We prove a version of the non-commutative Khintchine inequalities (originally due to Lust--Piquard) that is valid for the Lorentz spaces $L_{p,q}(\tau)$
associated to a non-commutative measure $\tau$, simultaneously for the whole range $1\le p,q< \infty$, regardless whether
$p<2 $ or $p>2$. Actually, the main novelty is the case $p=2,q\not=2$. We also prove a certain simultaneous decomposition property
for the operator norm and the Hilbert-Schmidt one.

\end{abstract}

\section{Introduction}

Let $B(\ell_2)$ denote  the space of all bounded operators on $\ell_2$.
Let $R\subset B(\ell_2)$ (resp.\ $C\subset B(\ell_2)$) be the row (resp. column) operator spaces, defined by $R = \ovl{\text{span}}[e_{1j}\mid j\ge 1]$ (resp.\ $C = \ovl{\text{span}}[e_{i1}\mid i\ge 1]$). 
The couple   $(R,C)$ plays an important r\^ole in Operator space theory. In particular, it is known that the complex interpolation space $(R,C)_{1/2}$ coincides with the (self-dual) operator space $OH$. See \cite{P2} for details. We refer to \cite{Xu} for the real interpolation method in the operator space framework. In particular, Xu proved in \cite{Xu} that $(R,C)_{1/2,2}$ is completely isomorphic to $OH$.

This paper studies three problems concerning real interpolation for several pairs of Banach spaces associated to $(R,C)$.

 In \S 3, we consider the pair $(M(R),M(C))$ when $M=B(\ell_2)$. The space $M(R)$ consists of those $x=(x_n)$ with $x_n\in B(\ell_2)$ such that $\sum x_nx^*_n$ converges in the weak operator topology (w.o.t.\ in short) and $\|x\|_{M(R)} \overset{\text{def}}{=} \|(\sum x_nx^*_n)^{1/2}\|$. Then $M(C)$ is formed of those $x=(x_n)$ such that $(x^*_n)\in M(R)$ with norm $\|x\|_{M(C)} \overset{\text{def}}{=} \|(\sum x^*_nx_n)^{1/2}\|$.

The main result of \S 3 is an equivalent expression for the $K$-functional for this pair $(M(R),M(C))$. Our result extends to more general (semi-finite) von~Neumann algebras. As an application we can describe the interpolation space $X(\theta) = (M(R),M(C))_{\theta,\infty}$ for $0<\theta<1$. We find that if $x$ is in the latter space, then $\|x\|^2_{X(\theta)}$ is equivalent to the norm of the associated completely positive map $T_x\colon \ T\mapsto \sum x_nTx^*_n$ as an operator of ``very weak type $(p,p)$'' on the $L_p$-space associated to the trace of $M$ with $p=1/\theta$. The analogous result for the complex interpolation method was obtained in our previous works (see \cite{P3,P4}).
Our result can be interpreted as a description of the operator space structure of $(R,C)_{\theta,\infty}$ in the sense of \cite{Xu}. 
Our approach is based on a non-commutative version of a lemma originally due to Varopoulos,
that we extended with a different proof in a separate paper \cite{P6}. 

In \S 4, we present a version of the non-commutative Khintchine inequalities (originally due to Lust--Piquard \cite{LP1}) that is valid for the Lorentz spaces $L_{p,q}(\tau)$ associated to $(M,\tau)$. 
This provides an equivalent for the average over all signs
of  the norm in  $L_{p,q}(\tau)$ of a series  of the form $\sum \pm     x_n$     ($x_n\in  L_{p,q}(\tau)$).
The main interest of our result is the case of $L_{2,q}(\tau)$ which seemed out of reach of previous works (see \cite{Ji}). Here again our study concentrates on a pair of Banach spaces, but this time it is the pair $(A_0,A_1)$ where $A_0 = M(R)\cap M(C)$ and where $A_1$ is the natural predual of $M(R)\cap M(C)$, that we describe as the sum of the preduals of $M(R)$ and $M(C)$ and we denote it by $A_1 = M_*(R)+M_*(C)$.

In \S 5, we study another pair, namely the pair $(A_0,A_2)$ where $A_2 =(A_0,A_1)_{1/2,2}$. When $M=B(\ell_2)$, the space $A_2$ is nothing but $\ell_2(S_2)$ where $S_2$ is the Hilbert--Schmidt class. We formulate our result using the notions of ``$K$-closed'' and ``$J$-closed'' introduced in \cite{P7}, that are isomorphic versions of Peetre's notion of ``subcouple''. To give a more concrete statement, the following  can be viewed as the main point of \S 5:

There is a constant $c$ such that for any $x$ in $(M(R)+M(C))\cap\ell_2(S_2)$ there is a decomposition $x=x_1+x_2$ such that we have simultaneously
\begin{align}
\|x_1\|_{M(R)} + \|x_2\|_{M(C)} & \le c\|x\|_{M(R)+M(C)}\\
\|x_1\|_{\ell_2(S_2)}  + \|x_2\|_{\ell_2(S_2)} &\le c\|x\|_{\ell_2(S_2)}.
 \end{align}
In our more abstract terminology, this says that the pair $(A_0,A_2)$  is $J$-closed when viewed
as sitting inside (via the diagonal embedding)
the pair  $(M(R)\oplus M(C),\ell_2(S_2)\oplus\ell_2(S_2) ) $.
This is analogous to what was proved in \cite{KLW} (resp. \cite{P7}) for the pair $(H^\infty,H^2)$ inside $(L^\infty,L^2)$
 (resp. $(H^\infty(B(\ell_2)),H^2(S_2))$ inside $(L^\infty(B(\ell_2)),L^2(S_2))$

 In \S 6, we briefly include a comparative discussion of  what the non-commutative Khintchine inequalities
become in free probability and how it relates to our interpolation problems.
\vspace{.3in}

\n {\bf Acknowledgment.} I am grateful to Quanhua Xu and Yanqi Qiu for stimulating discussions.

\section{Notation and background}\label{sec1-}

We will use the real interpolation method. We refer to \cite{BL} for all undefined terms.
We just recall that if $(A_0,A_1)$ is a compatible couple of Banach spaces,
then for any $x\in A_0+A_1$ the $K$-functional is defined by
$$\forall t>0\qquad K_t(x;A_0,A_1)= \inf
\big({\|x_0\|_{A_0}+t\|x_1\|_{A_1}\ | \ x=x_0+x_1,x_0\in
A_0,x_1\in A_1}). $$
Recall that the (``real" or ``Lions-Peetre" interpolation) space
$(A_0,A_1)_{\theta,q}$ is defined, for $0<\theta<1$ and $1\le q\le \infty$, as the space of all $x$
in $A_0+A_1$ such that $\|x\|_{\theta,q} <\infty$ where
 $$\|x\|_{\theta,q} =(\int{(t^{-\theta}K_t(x,A_0,A_1))^{q}
dt/t})^{1/q} ,$$
with the usual convention when $q=\infty$.

Let $M$ be a von Neumann algebra equipped with a semi-finite faithful normal trace $\tau$. 
The basic example is $M=B(\ell_2)$, equipped with the  usual trace, and, to improve readability,  we will present most of our results in this special case
with mere  indications for the extension to the general case.

Let $M_*$ be the predual of $M$.
As is well known (see e.g. \cite{PX})
$M_*$ can be identified with the non-commutative $L_1$-space
associated to $\tau$ usually denoted by $L_1(\tau)$. When $M=B(\ell_2)$, $M_*$ is the classical ``trace class" $S_1$.
More generally, for any $1\le p<\infty$ we denote by $L_p(\tau)$
the associated non-commutative $L_p$-space. By convention we set $L_\infty(\tau)=M$.
When $M=B(\ell_2)$, $M_*$ (resp. $L_p(\tau)$) is the classical ``trace class" $S_1$ (resp. the Schatten class $S_p$).
See e.g.  \cite{FK} or  \cite{PX} for more information on  non-commutative $L_p$-spaces. 
We will first mainly use the case $p=2$ and we denote its norm simply by $\|.\|_2$.

We always denote by $p'$ the conjugate of $1\le p\le \infty $ defined by $p^{-1}  +  p'^{-1}=1$.

We denote by ${\cl P}(M)$ or simply by ${\cl P}$ the set of all (self-adjoint) projections in $M$.

We denote by $M(R)$ (resp.\ $M(C)$) the space of sequences $x=(x_n)$ with coefficients $x_n\in M$ such that $(\sum^\infty_1 x_nx^*_n)^{1/2}\in M$ (resp.\ $(\sum^\infty_1 x^*_nx_n)^{1/2}\in M$). Here we implicitly assume that the series $\sum^\infty_1 x_nx^*_n$ (resp.\ $\sum^\infty_1 x^*_nx_n$) converge in (say) the  w.o.t. . We equip $M(R)$ (resp.\ $M(C)$) with the natural norm
\[
\|x\|_{M(R)} = \left\|\left(\sum x_nx_n^*\right)^{1/2}\right\|_M \left(\text{resp. } \|x\|_{M(C)} = \left\|\left(\sum x^*_nx_n\right)^{1/2}\right\|_M\right).
\]
Note that $M(R)$ (resp.\ $M(C)$) can be identified with $M\ovl\otimes R$ (resp.\ $M\ovl\otimes C$), i.e.\ the weak-$*$ closure of $M\otimes R$ (resp.\ $M\otimes C$) in the (von Neumann sense) tensor product $M\ovl\otimes B(\ell_2)$.

We will consider $(M(R), M(C))$ as an interpolation couple in the obvious way using the inclusions $M(R)\subset M^{\bb N}$, $M(C)\subset M^{\bb N}$.

Similarly, we denote by $M_*(R)$ (resp.\ $M_*(C)$) the space of sequences $x=(x_n)$ with coefficients $x_n\in M_*$ such that  $(\sum^\infty_1 x_nx^*_n)^{1/2}\in M_*$ (resp.\ $(\sum^\infty_1 x^*_nx_n)^{1/2}\in M_*$). Here we   assume that the sequence $(\sum^N_1 x_nx^*_n)^{1/2}$ (resp.\ $(\sum^N_1 x^*_nx_n)^{1/2}$) norm-converges   in $M_*$ when $N\to \infty$.
We equip $M_*(R)$ (resp.\ $M_*(C)$) with the natural norm
\begin{equation}\label{eqs1}
\|x\|_{M_*(R)} = \left\|\left(\sum x_nx_n^*\right)^{1/2}\right\|_{M_*} \left(\text{resp. } \|x\|_{M_*(C)} = \left\|\left(\sum x^*_nx_n\right)^{1/2}\right\|_{M_*}\right).
\end{equation}

Note that $M(R)=M_{*}(C)^*$ and $M(C)=M_{*}(R)^*$ isometrically with respect to the duality defined by $\langle x, y\rangle= \sum {\tau}(x_n  y_n)$.
(Equivalently, $\overline{M(R)}=M_{*}(R)^*$ 
and $\overline{M(C)}=M_{*}(C)^*$ with respect to the duality defined by $\langle x, \overline y\rangle= \sum {\tau}(x_n  y^*_n)$,
with the bar denoting complex conjugation.)

\section{$K$-functional between $R$ and $C$}\label{sec1}

Our main result is:

\begin{thm}\label{thm1.1}
For any $a = (a_n)\in M(R) + M(C)$ we have
\begin{equation}
k_t(a) \le K_t(a; M(R), M(C)) \le 2k_t(a)\tag*{$\forall t>0$}
\end{equation}
where
\[
k_t(a) = \sup\left\{\left(\sum \|Pa_nQ\|^2_2\right)^{1/2} \max \{\tau(P)^{1/2}, t^{-1}\tau(Q)^{1/2}\}^{-1} \ \Big| \ P,Q\in {\cl P}\right\}.
\]
\end{thm}

 \begin{rk} The following was  pointed out by Q. Xu (by a modification of the proof of
 Lemma \ref{lem-d} below). Let $\widehat k_t(a)$
 be the same as $k_t(a)$ except that, when $t> 1 $ (resp. $t<1$) we restrict to 
 pairs of projections $P,Q$ such that $P\le Q$ (resp. $Q\le P$), and when $t=1$
 we restrict to pairs such that $P=Q$.
 Then
 $$\widehat k_t(a) \le   k_t(a) \le 2^{1/2}\widehat k_t(a).$$
 We merely indicate a quick  argument for $t>1$.
 Let $Q'=P \vee Q$. Then $P\le Q'$ and $\tau(Q')\le \tau(P)+\tau(Q)$.
 With the above notation we have
 $\left(\sum \|Pa_nQ\|^2_2\right)^{1/2}\le \left(\sum \|Pa_nQ'\|^2_2\right)^{1/2}$
 and also 
 $\tau(P)^{1/2}\vee t^{-1}\tau(Q')^{1/2} \le \tau(P)^{1/2}\vee  t^{-1} (\tau(P)+\tau(Q))^{1/2}\le (t^{-2}+1)^{1/2} (\tau(P)^{1/2}\vee   t^{-1}\tau(Q)^{1/2}$
 and since 
 $(t^{-2}+1)^{1/2}\le 2^{1/2}$ we obtain $k_t(a) \le 2^{1/2}\widehat k_t(a).$
 We leave the other cases to the reader. 

 \end{rk}
\begin{proof}[First part of the proof of Theorem \ref{thm1.1}]
Consider $a^0\in M(R),\ P,Q\in {\cl P}$. We have
\begin{align*}
\sum\|Pa^0_nQ\|^2_2 &= \sum\tau(Pa^0_nQa^{0*}_nP) \le \sum\tau(Pa^0_na^{0*}_nP)\\
&\le \left\|\sum a^0_na^{0*}_n\right\|\tau(P)\\
\intertext{and hence}
\left(\sum\|Pa^0_nQ\|^2_2\right)^{1/2} &\le \|a^0\|_{M(R)} \tau(P)^{1/2}.
\end{align*}
Similarly for any $a^1\in M(C)$ we have
\[
\left(\sum\|Pa^1_nQ\|^2_2\right)^{1/2} \le \|a^1\|_{M(C)} \tau(Q)^{1/2}.
\]
Therefore if $a = a^0+a^1$ we find by the triangle inequality
\begin{align*}
\left(\sum\|Pa_nQ\|^2_2\right)^{1/2}&\le \|a^0\|_{M(R)} \tau(P)^{1/2} + \|a^1\|_{M(C)} \tau(Q)^{1/2}\\
&\le (\|a^0\|_{M(R)} + t\|a^1\|_{M(C)}) \max\{\tau(P)^{1/2}, t^{-1}\tau(Q)^{1/2}\}.
\end{align*}
So we obtain
\[
k_t(a) \le K_t(a; M(R), M(C)).
\] \end{proof}

To prove the converse we will use duality, via the following Lemma. 
\begin{lem}\label{lem-d}
Let $x\in M_*(R)\cap M_*(C)$ be such that
\[
J_t(x) \overset{\text{\rm def}}{=} \max\left\{\left\|\left(\sum x^*_n x_n\right)^{1/2}\right\|_1, \frac1t \left\|\left(\sum x_n x^*_n\right)^{1/2}\right\|_1\right\}\le 1.
\]
Let ${\cl C}_t$ be the subset of $M_*(R)\cap M_*(C)$ formed of all sequences $\chi=(\chi_n)$ of the form
\[
\chi_n = Qy_nP \cdot (\tau(P)+t^{-2}\tau(Q))^{-1/2}
\]
with $\sum\|y_n\|^2_2\le 1$ and where $P,Q$ are   commuting projections.
Then $x\in 2~\ovl{\text{\rm conv}}({\cl C}_t)$ where the closure is in $M_*(R)\cap M_*(C)$.
\end{lem}
We will need the following simple

\begin{lem}\label{elem}
Let $\varphi\colon \ [1,\ldots, N]^2\to {\bb R}_+$ be defined by $\varphi(i,j) = g(i)\wedge f(j)$ where $g\ge 0$ and $f\ge 0$ satisfy $\sum f(j)\le 1$ and $\sum g(i)\le t^2$. Then $\varphi \in \text{\rm conv}(\Phi)$ where $\Phi$ is the set of functions on $[1,\ldots, N]^2$ of the form
\begin{equation}\label{re-eq3}
\frac{1_{E\times F}}{\frac1{t^2} |E| + |F|}
\end{equation}
where $E,F\subset [1,\ldots, N]^2$ are arbitrary subsets.
\end{lem}

\begin{proof}
We may write
\begin{align*}
\varphi &= \int^\infty_0 1_{\{\varphi>c\}} dc = \int^\infty_0 1_{\{g>c\}\times \{f>c\}} dc\\
&= \int^\infty_0 m(c) \frac{1_{\{g>c\}\times \{f>c\}}}{m(c)}dc
\end{align*}
where $m(c) = \frac1{t^2}|\{ g>c\}| + |\{f>c\}|$. But since $\int^\infty_0 m(c)dc = \frac1{t^2} \sum g(i)+ \sum f(j)\le 2$ the Lemma follows.
\end{proof}

\begin{rk}
A simple verification shows that if $\varphi\in \Phi$ is of the form \eqref{re-eq3}, we have $\sup_j \varphi(i,j) = 1_{\{i\in E\}}(t^{-2} |E|+|F|)^{-1}$ and $\sup_i\varphi(i,j) = 1_{\{j\in F\}}(t^{-2}|E|+|F|)^{-1}$. Therefore we find
\begin{equation}\label{re-eq4}
t^{-2} \sum\nolimits_i \sup\nolimits_j \varphi(i,j) + \sum\nolimits_j \sup\nolimits_i \varphi(i,j) \le 1.
\end{equation}
\end{rk}

\begin{rk}
Let $(\Omega,\mu)$ and $(\Omega',\mu')$ be measure spaces. Let $f\colon \ \Omega'\to {\bb R}_+$ and $g\colon  \ \Omega\to {\bb R}_+$ be   step functions. Assume that $\int f\ d\mu'\le 1$ and $\int g\ d\mu \le t^2$ $(t>0)$. Let $\varphi(\omega,\omega') = g(\omega)\wedge f(\omega')$ on $\Omega\times\Omega'$. Then $\varphi\in  2\text{conv}(\Phi)$ where $\Phi$ is the set of functions of the form
\[
\varphi = \frac{1_{E\times F}}{t^{-2}\mu(E) + \mu'(F)},
\]
 where $E\subset\Omega$ and $F\subset\Omega'$ are arbitrary measurable subsets. 
 \end{rk}

\begin{proof}[Proof of Lemma \ref{lem-d}]
As is well known, if we truncate the sequence $(x_n)$ and replace it by $x(N)$ defined by $x_n(N) = x_n1_{\{n\le N\}}$, then $\|x(N)-x\|_{M_*(R)}\to 0$ and similarly for $M_*(C)$. 
Indeed, since for all $N\le m$ the norms in $M_*(R)$ and $M_*(R)$ both satisfy
$$\| x(N)-x(m)\|^2 + \|x(N)\|^2\le \|x(m)\|^2$$
and since $\|x(m)\|  \to \|x\|$, this fact follows easily.\\
Thus it suffices to prove the Lemma for finite sequences $x = (x_1,\ldots, x_N)$. Let us first assume that $M=B(\ell_2)$, $M_*=S_1$ (trace class) and $\tau$ is the ordinary trace on $S_1$.  Assume $J_t(x)<1$. Let $f = (\sum x^*_nx_n)^{1/2}$ and $g = t(\sum x_nx^*_n)^{1/2}$. We have ${\rm tr}(f)<1$, ${\rm tr}(g)<t^2$ and moreover if $a_n=g^{-1}x_n$, $b_n=x_nf^{-1}$ then $\sum a_na^*_n = g^{-1}t^{-2}g^2g^{-1} = t^{-2}$, and similarly $\sum b^*_nb_n=1$, so we have
\begin{equation}\label{re-eq1}
\left\|\left(\sum a_na^*_n\right)^{1/2}\right\|\le t^{-1}, \quad \left\|\sum b^*_nb_n\right\|^{1/2}\le 1,
\end{equation}
with $x_n=ga_n =b _nf$.

Note that by a simple perturbation argument we may assume $f>0$ and $g>0$ so that $f$ and $g$ are invertible. 

Let us now consider the matrix representation of $x_n$ with respect to the bases that diagonalize respectively $f$ (for the column index) and $g$ (for the row index). We have then
\[
x_n(i,j) = g_ia_n(i,j) = b_n(i,j)f_j.
\]
Moreover we know from \eqref{re-eq1} that $\forall\xi\in \ell_2$
\[
\sum \|a^*_n\xi\|^2 \le t^{-2}\|\xi\|^2\quad \text{and}\quad \sum\nolimits\|b_n\xi\|^2 \le \|\xi\|^2
\]
and hence taking for $\xi$ either the $i$-th or the $j$-th basis vector we find
\begin{equation}\label{re-eq2}
\sup_i \sum\nolimits_n \sum\nolimits_j|a_n(i,j)|^2\le t^{-2}\quad \text{and}\quad \sup_j \sum\nolimits_n \sum\nolimits_i |b_n(i,j)|^2\le  1.
\end{equation}
Let $\gamma_n(i,j) = \frac1{g_i\wedge f_j} (x_n(i,j))$. Then $|\gamma_n(i,j)| \le \max\{|a_n(i,j)|,| b_n(i,j)|\} \le |a_n(i,j)| + |b_n(i,j)|$ and
$x_n(i,j) = (g_i\wedge f_j) \gamma_n(i,j)$.
By Lemma \ref{elem},
 since $x_n(i,j) = g_i\wedge f_j$ $\gamma_n(i,j)$, we know that $x/2$ is in the convex hull of
\[
(\varphi(i,j)\gamma_n(i,j))_n = (\varphi(i,j)^{1/2} \varphi(i,j)^{1/2} \gamma_n(i,j))_n
\]
where $\varphi\in\Phi$. We then note that if
\[
y_n(i,j) = \varphi(i,j)^{1/2} \gamma_n(i,j)
\]
then using \eqref{re-eq2} and \eqref{re-eq4} we have since $|\gamma_n(i,j)| \le |a_n(i,j)| \vee |b_n(i,j)|$
\begin{align*}
\sum\nolimits_n \|y_n\|^2_2 &= \sum\nolimits_{n,i,j} |\xi_n(i,j)|^2\\
&\le \sum\nolimits_n \sum\nolimits_{ij} \varphi(i,j)(|a_n(i,j)|^2 + |b_n(i,j)|^2)\\
&\le \sum\nolimits_i \sup\nolimits_j\varphi(i,j) t^{-2} + \sum\nolimits_j \sup\nolimits_i\varphi(i,j)\\
&\le 1.
\end{align*}
Thus we find that $x/2$ can be written as a convex combination of elements of the form
\[
(\varphi(i,j)^{1/2} y_n)_n
\]
with $\varphi\in \Phi$ and $\sum\|y_n\|^2_2\le 1$. Let $Q,P$ be the projections associated to $1_E$ and $1_F$. Then
\[
[\varphi(i,j)^{1/2} y_n] = (t^{-2} \text{ tr } Q + \text{ tr } P)^{-1/2}Qy_nP.
\]
 This completes the proof  in the case $M=B(H)$.
The case of a general semi-finite von~Neumann algebra $M\subset B(H)$ can easily be reduced (by density) to the case when $M$ is finite.
 %We can also assume $H$ separable (on $M$ countably generated).
  In that case, the densities $f$ and $g$ in the preceding argument can be replaced by $f_\vp = f+\vp1$ and $g_\vp = g+\vp1$ in order to obtain $f,g$ invertible. We are thus left with a finite sequence $x_1,\ldots, x_N$ in $M$ and $f,g\ge 0$ in $M_*$ invertible such that $\tau(f)<1$ and $\tau(g)<t^2$ and moreover
\begin{equation}\label{re-eq5}
x_n = ga_n = b_nf\qquad (n\le N) 
\end{equation}
with $a_n,b_n\in M$ such that \eqref{re-eq1} holds. Just like we do for functions and step functions we may approximate $f,g$ by elements of the form $\sum^k_1 f_jP_j$ and $\sum^k_1 g_iQ_i$ where $(P_i)$ (resp.\ $Q_j)$) are orthogonal projections in $M$ with $\sum P_i = \sum Q_j=1$. This modification leads by \eqref{re-eq5} to a perturbation of $x_n$ so it suffices to complete the proof for this special case. If we then denote $x_n(i,j) = P_i x_n Q_j$, and similarly for $a_n$ and $b_n$ we can essentially repeat the preceding argument   using the remark following the above Lemma with the measures $\mu = \sum g_j\tau(Q_j)\delta_j$ and $\mu' = \sum f_i\tau(P_i)\delta_i$.
\end{proof}

 \begin{proof}[End of the proof of Theorem \ref{thm1.1}]
Assume $k_t(a)\le 1$, so that $\forall P,Q\in {\cl P}$
\[
\sum\|Pa_nQ\|^2_2 \le {\tau}(P) \vee t^{-2} {\tau} (Q).
\]
This implies by Cauchy--Schwarz
\begin{align*}
\left|\sum {\tau}(Pa_nQy_n)\right| &\le ({\tau}(P)\vee t^{-2} {\tau}(Q))^{1/2} \left(\sum\|y_n\|^2_2\right)^{1/2}\\
&\le (t^{-2} {\tau}(Q) + {\tau}(P))^{1/2} \left(\sum\|y_n\|^2_2\right)^{1/2} 
\end{align*}
and hence by Lemma \ref{lem-d}
\[
\left|\sum {\tau}(a_nx_n)\right| \le 2J_t(x),
\]
then by duality we conclude
\[
K_t(a; M(R), M(C)) \le 2.
\]
\end{proof}

\begin{rk}
The preceding proof reveals the following slightly surprising fact: \ Consider a sequence $x=(x_n)$ of operators $x_n\in B(\ell_2)$, with say $H=\ell_2$. Assume that for any pair of orthonormal basis $(e_i)$ $(f_j)$ in $H$ the matrix $a_{ij} = \langle e_i, x_nf_j\rangle$ belongs to $\ell_\infty(i; \ell_2(n,j)) + \ell_\infty(j;\ell_2(n,i))$. Then $x=(x_n)\in M(R)+ M(C)$ with $M=B(\ell_2)$. Indeed, if $E$ (resp.\ $F$) is a finite subset of ${\bb N}$, and if $P$ (resp.\ $Q$) is the orthogonal projection into $\text{span}(E)$ (resp.\ $\text{span}(F)$), then
\[
\left(\sum \|Pa_nQ\|^2\right)^{1/2} (\max(\tau(P)^{1/2}, \tau(Q)^{1/2}))^{-1}\\
= \left(\sum\nolimits_{E\times F} \sum\nolimits|a_n(i,j)|^2 \max\{|E|, |F|\}^{-1}\right)^{1/2}.
\]
So Theorem~\ref{thm1.1} (compare with the Varopoulos Lemma in \cite{Va}) implies the preceding fact. Moreover the converse implication also holds since $x\in M(R)$ (resp.\ $M(C)$) implies $a\in\ell_\infty(i;\ell_2(n,j))$ (resp.\ $\ell_\infty(j; \ell_2(n,i)$).
\end{rk}

\begin{rk}
Let $N\subset M$ be a von Neumann subalgebra such that $\tau_{|N}$ is still semi-finite, so that the conditional expectation ${\bb E}\colon\ M\to N$ is well defined. It is easy to check that the main result remains valid if we replace the norms of $M(R)$ (resp.\ $M(C)$) by their conditional versions:
\[
\left\|\sum {\bb E}(x_nx^*_n)\right\|^{1/2} \quad \left(\text{resp.}\left\|\sum {\bb E}( x^*_nx_n)\right\|^{1/2}\right).
\]
The formula for $k_t$ now has to be modified by restricting $p,q$ to lie in $N$.
\end{rk}

To put the next corollary in proper perspective, recall that, according to \cite{P3}, the elements $a=(a_n)$ in the complex interpolation space $(M(R),M(C))^\theta$ are precisely those such that the operator
\[
T_a\colon \ x\mapsto \sum a_nxa^*_n 
\]
is bounded on $L_p(\tau)$, where $p=1/\theta$.

 In the commutative case, a bounded operator $T\colon \ L_p(\Omega_1,\mu_1)\to L_p(\Omega_2,\mu_2)$ is called of strong type $p$, and the classical Riesz interpolation theorem says that, in the complex case, if $1\le p_0<p_1\le \infty$ and $T$ is of strong type $p_j$ for $j=0,1$ then $T$ is of strong type $p$ for any  intermediate $p$ such that $p_0<p<p_1$. The latter theorem is the founding result for the ``complex interpolation method'' (see \cite{BL}), while the classical Marcinkiewicz theorem is the basis for the ``real method''. In that context, an operator that is bounded from $L_{p,1}(\mu_1)$ to $L_{p,\infty}(\mu_2)$ is called of very weak type $p$. The generalized version of Marcinkiewicz theorem then says that, if $T$ is of very weak type $p_j$ for $j=0,1$, then $T$ is of strong type $p$ for any $p_0<p<p_1$. 

Let $(\Omega,\mu)$ be a measure space. Recall that the ``weak $L_p$'' space $L_{p,\infty}(\mu)$ is formed of all measurable functions $f\colon \ \Omega\to {\bb R}$ such that
\[
\|f\|_{p,\infty} = \sup_{c>0} (c^p\mu\{|f|>c\})^{1/p} < \infty.
\]
When $p>1$, the quasi-norm $\|\cdot\|_{p,\infty}$ is equivalent to the following norm
\begin{equation}\label{99}
\|f\|_{[p,\infty]} = \sup\left\{\int_E |f| \frac{d\mu}{\mu(E)^{1/p'}}\ \Big|\ E\subset \Omega\right\}.
\end{equation}
When $p>1$,  $L_{p',\infty}(\mu)$ is the dual of the ``Lorentz space'' $L_{p,1}(\mu)$ that can be defined
(see e.g. \cite{BL}) as formed of those $f$ such that
\begin{equation}\label{100}
[f]_{p,1} = \int^\infty_0 \mu\{|f|>c\}^{1/p} \ dc < \infty.
\end{equation}
Using the generalized $s$-numbers from \cite{FK}, it is easy to define
the spaces  $L_{p,\infty}(\tau)$ and $L_{p,1}(\tau)$ (or more generally  $L_{p,q}(\tau)$ for $1\le q\le \infty$)
associated to $(M,\tau)$. The simplest way to describe those is  as follows. Given a $\tau$-measurable
operator $x$ (in the sense of \cite{FK}), let $M_{|x|}\subset M$ denote the von Neumann subalgebra 
generated by the spectral projections of  $|x|$. Then $M_{|x|}\simeq L_\infty(\Omega_{|x|},\mu_{|x|})$
for some measure space $(\Omega_{|x|},\mu_{|x|}) $  in such a way that the restriction of $\tau$ to $M_{|x|}$ coincides with $\mu_{|x|}$ in this identification. Moreover the  space 
of scalar valued measurable functions (that are bounded outside a set of finite measure) $ L_0(\Omega_{|x|},\mu_{|x|})$ can be identified
with that of $\tau$-measurable operators affiliated with $M_{|x|}$.
The space $L_{p,\infty}(\tau)$   (resp.  $L_{p,q}(\tau)$ for $1\le q\le \infty$) is then formed
of those $x$ such that, in the latter identification, $|x|\in L_{p,\infty}(\mu_{|x|})$ (resp.  $L_{p,q}(\mu_{|x|})$.
The duality extends to this setting: we have $L_{p',\infty}(\tau)=L_{p,1}(\tau)^*$ for any $1\le p< \infty$, see \cite{FK,PX}.

The following two statements appear as analogues for real interpolation of the complex case already treated in \cite{P0,P3}. 

\begin{cor}\label{cor1}
Let $0<\theta<1$. An element $a = (a_n)$ $(x_n\in M)$ belongs to the space $(M(R),M(C))_{\theta,\infty}$ iff the mapping
\[
T_\alpha\colon \ x\to \sum a_nx a^*_n
\]
is bounded from $L_{p,1}(\tau)$ to $L_{p,\infty}(\tau)$ where $\frac1p = \frac{1-\theta}\infty +\frac\theta1$, i.e.\ $p=1/\theta$. Moreover, the norm in $(M(R),M(C))_{\theta,\infty}$ is equivalent to  $a\mapsto \|T_a\colon \ L_{p,1}(\tau)\to L_{p,\infty}(\tau)\|^{1/2}$.
\end{cor}

\begin{proof}
Recall
\begin{equation}\label{re-eq10a}
\forall a_0,a_1>0\qquad\qquad a^{1-\theta}_0a^\theta_1 = \inf_{t>0} \{(1-\theta) a_0t^\theta + \theta a_1t^{\theta-1}\}.
\end{equation}
By definition
\[
\|a\|_{\theta,\infty} = \sup\nolimits_{t>0} t^{-\theta} K_t(a; M(R),M(C)).
\]
By   Theorem \ref{thm1.1} this is equivalent to $\sup_{t>0} t^{-\theta}k_t(a)$. Note that $(1-\theta)\xi + \theta\eta \le \max(\xi,\eta) \le \max\{\theta^{-1},(1-\theta)^{-1}\}( (1-\theta)\xi +\theta\eta)$ $\forall\xi,\eta>0$. Therefore
\[
\sup\nolimits_{t>0} t^{-\theta}(\max(\tau(P)^{1/2}, t^{-1}\tau(Q)^{1/2}))^{-1}
\]
is equivalent to 
\[
(\inf\nolimits_{t>0} (1-\theta)t^{2\theta} \tau(P) + \theta t^{2\theta-2} \tau(Q))^{-1/2}
\]
or equivalently (by \eqref{re-eq10a}) to $\inf_{s>0}((1-\theta) s^\theta \tau(P) + \theta s^{\theta-1} \tau(Q))^{-1/2}=(\tau(P)^{1-\theta} \tau(Q)^\theta)^{-1/2}$. Thus we find that $\|a\|^2_{\theta,\infty}$ is equivalent to
\begin{align*}
&\sup\left\{\sum \|Pa_nQ\|^2_2 (\tau(P)^{1-\theta} \tau(Q)^\theta)^{-1}\mid P,Q\in {\cl P}\right\}\\
= &\sup\left\{\sum \tau(Pa_nQa^*_n) (\tau(P)^{1-\theta} \tau(Q)^\theta)^{-1} \mid P,Q\in {\cl P}\right\}\\
= &\sup\{\langle T_a(Q\tau(Q)^{-\theta}) , P\tau(P)^{-(1-\theta)}\rangle \mid P,Q\in {\cl P}\}.
\end{align*}
This last expression is equivalent to
\[
\sup\{\langle T_a(x),y\rangle \mid x\in B_{L_{p,1}(\tau)}, y\in B_{L_{p',1}(\tau)}\}.
\]
Indeed, using convex combinations of elements of the form $Q\tau(Q)^{-\theta}$ we obtain the case of $x\ge 0$ in $B_{L_{p,1}(\tau)}$; then the decomposition $x=x_1-x_2+i(x_3-x_4)$ yields the general case up to a factor 4. The same reasoning applies to $y$. So we conclude that $\|a\|^2_{\theta,\infty}$ is equivalent to $\|T_a\colon \ L_{p,1}(\tau)\to L_{p,\infty}(\tau)\|$.
\end{proof}
\begin{rk}
In the particular  when $M=B(\ell_2)$ or $M_n$ with $n$ arbitrary, the preceding corollary yields a  description of the operator space structure of   $(R,C)_{\theta,\infty}$ according to Xu's definition in \cite{Xu}.
\end{rk}
 Let $1<p\le\infty$. When $p=\infty$, by convention we identify $L_{p,\infty}(\tau)$ with $M$.
Let $M(R;p,\infty)$ denote the space of sequences $a=(a_n)$ with $a_n\in L_{p,\infty}(\tau)$ such that $(\sum a_na^*_n)^{1/2}\in L_{p,\infty}(\tau)$, equipped with the ``norm'' $\|a\| = \|(\sum a_na^*_n)^{1/2}\|_{p,\infty}$. Similarly, we define $M(C;p,\infty) = \{a=(a_n)\mid (a^*_n)\in M(R;p,\infty)\}$ with $\|a\|_{M(C;p,\infty)} = \|(\sum a^*_na_n)^{1/2}\|_{p,\infty}$.

Using a simple non-commutative adaptation of the results in \cite{P6} along the lines of the proof of Theorem~\ref{thm1.1}, we find

\begin{thm}\label{thm2b}
Let $2<p_0,p_1\le\infty$. Let $a=(a_n)$ be a sequence with $a_n\in L_{p_0,\infty}(\tau) + L_{p_1,\infty}(\tau)$ for all $n$. To abbreviate, we set
\begin{align*}
K_t(a) &= K_t(a; M(R;p_0,\infty), M(C; p_1,\infty)),\quad \text{and}\\
k_t(a) &= \sup\left\{\left(\sum\|Pa_nQ\|^2_2\right)^{1/2} \max\{\tau(P)^{\alpha_0}, t^{-1}\tau(Q)^{\alpha_1}\}^{-1} \mid P,Q\in {\cl P}\right\}
\end{align*}
where $\alpha_0 = \frac12-\frac1{p_0}, \alpha_1 = \frac12 - \frac1{p_1}$. Then there are positive constants $c$ and $C$ (depending only on $p_0,p_1$) such that
\begin{equation}
ck_t(a) \le K_t(a) \le Ck_t(a).\tag*{$\forall t>0$}
\end{equation}
\end{thm}

\begin{proof}
Let $(\Omega,\mu)$ be any measure space. For any measurable $f\colon \ \Omega\to {\bb R}_+$ 
we have obviously $\|f\|^p_{p,\infty} = \|f^2\|^{p/2}_{p/2,\infty}$.
 Therefore, using \eqref{99} $\|f\|_{p,\infty}$ is equivalent to 
\[
\sup\limits_{E\in\Omega} (\int_E |f|^2\ d\mu)^{1/2} \mu(E)^{\frac1p-\frac12}.
\]
 Thus we may use on $M(R;p_0,\infty)$ the following equivalent norm
\[
\sup\left\{\left(\tau\left(\sum x_nx^*_nP\right)\right)^{1/2} \tau(P)^{-\alpha_0}\mid P\in {\cl P}\right\}.
\]
Similarly, we may equip $M(C;p_1,\infty)$ with the equivalent norm
\[
\sup\left\{\left(\tau \left(\sum x^*_nx_nQ\right)\right)^{1/2} \tau(Q)^{-\alpha_1}\mid Q\in{\cl P}\right\}.
\]
Using these norms we find $k_t(a) \le K_t(a)$ by the same reasoning as for Theorem~\ref{thm1.1}. To prove the converse, we use duality again and mimic the proof of Theorem~\ref{thm1.1} using as a model the results presented in \cite{P6} for the commutative case.
\end{proof}

\begin{cor}\label{cor6}
Consider $2<p_0,p_1 \le\infty$ and $0<\theta<1$. Let $a=(a_n)_n$ be a sequence in $L_{p_\theta,\infty}(\tau)$ where $\frac1{p_\theta} = \frac{1-\theta}{p_0} + \frac\theta{p_1}$. Then $a=(a_n)$ belongs to the space $(M(R;p_0,\infty), M(C;p_1,\infty))_{\theta,\infty}$ iff the operator $T_a$ is bounded from $L_{r,1}(\tau)$ to $L_{s,\infty}(\tau)$ where $r,s$ are determined by $\frac{1-\theta}\infty + \frac\theta{p'_1} =\frac1r$ (i.e.\ $r=p'_1/\theta$) and $\frac{1-\theta}{p_0} + \frac\theta1 = \frac1s$ (i.e.\ $s=p_0 (1-\theta+\theta p_0)^{-1}$). Moreover the norm of $a$ in that space is equivalent to $\|T_a\colon \ L_{r,1}(\tau) \to L_{s,\infty}(\tau)\|^{1/2}$.
\end{cor}

\begin{proof}
Using the equivalent of $K_t$ found in Theorem~\ref{thm2b}, we obtain
\[
\sup t^{-\theta} K_t(a) \simeq \sup\left\{\left(\sum \|Pa_nQ\|^2_2\right)^{1/2} \tau(P)^{-\alpha_0(1-\theta)}\tau(Q)^{-\alpha_1\theta}\right\}.
\]
The unit ball for this last norm is characterized by
\begin{align*}
\sum\tau(Pa_nQa^*_n) &\le \tau(P)^{2\alpha_0(1-\theta)} \tau(Q)^{2\alpha_1\theta}\\
&\le \tau(P)^{(1-\theta)/p'_0} \tau(Q)^{\theta/p'_1}
\end{align*}
or equivalently
\[
\langle T_a(Q),P\rangle \le \tau(Q)^{1/r} \tau(P)^{1/s'}.
\]
As before, this implies for all $x,y\ge 0$
\[
|\langle T_a(x),y\rangle|\le \|x\|_{r,1} \|y\|_{s',1}
\]
and then using $x=x_1-x_2 + i(x_3-x_4)$ we can extend it to arbitrary elements up to an extra factor 4.
Thus we conclude by homogeneity
\[
\sup\nolimits_{t>0} t^{-\theta}K_t(a) \simeq \|T_a\colon \ L_{r,1}(\tau)\to L_{s,\infty}(\tau)\|^{1/2}.\eqno \qed
\]
\renewcommand{\qed}{}\end{proof}
\begin{rem}\label{rem7}
By \cite{Mali} for any interpolation pair $(B_0,B_1)$ we have
\[
(B_0\cap B_1, B_0+B_1)_{\theta,q} = \begin{cases}
                                    B_{\theta,q}\cap B_{1-\theta,q}&\text{if ~~$\theta\le 1/2$}\\
B_{\theta,q} +B_{1-\theta,q}&\text{if ~~$\theta\ge 1/2$}
                                    \end{cases}
\]
where $B_{\theta,q} = (B_0,B_1)_{\theta,q}$.
If we apply this to the specific pair $$B_0  = S_\infty[ R],\quad B_1= S_\infty[ C],  $$
the result can be interpreted   in terms of operator space
interpolation   in Xu's sense (see \cite{Xu}). This gives us that we have \emph{completely isomorphically}
\[
(R\cap C, R+C)_{\theta,\infty} = \begin{cases}
                                 C_{\theta,\infty} \cap R_{\theta,\infty}&\theta\le 1/2\\
C_{\theta,\infty}+R_{\theta,\infty}&\theta\ge 1/2
                                 \end{cases}
\]
where $C_{\theta,q} = (R,C)_{\theta,q}$ and $R_{\theta,q} = (C,R)_{\theta,q} = C_{1-\theta,q}$. Note that in particular, we have as operator spaces:
\[
(R\cap C, R+C)_{1/2,\infty} \simeq R_{1/2,\infty} \simeq C_{1/2,\infty}
\]
and similarly, by duality,  (see \cite[\S 4]{Xu})  for $(1/2,1)$.
\end{rem}
\section{Non-commutative Khintchine inequality in $L_{2,q} $ ($1\le q<\infty)$}

This section is motivated by \cite{Ji}. In   \cite{Ji},    martingale inequalities
extending those of \cite{PX} are proved for the non-commutative Lorentz spaces
$L_{p,q}(\tau)$ associated to a semi-finite trace with $p\not=2$. However,
the case  $p=2,q\not= 2$  cannot be treated by the interpolation arguments
used in  \cite{Ji}. In fact,  even the simpler case of the Khintchine inequality
is open. The problem is to find a ``nice" (similarly nice as in the case
of $L_p(\tau)$ presumably involving row and column norms) equivalent of the average over all signs $\vp_n=\pm 1$
of \begin{equation}\label{2q}
\left\|\sum \vp_n x_n\right\|_{2,q}\end{equation}
when $x_n\in L_{2,q}(\tau)$.
 In this section we present a partial solution, which 
has the advantage to be indeed a deterministic equivalent of 
 \eqref{2q}. We call it partial because there may be
a  more explicitly computable   equivalent for \eqref{2q}.

 \n {\bf Notation:} 
 Recall that  
 $S_p$ denotes the Schatten $p$-class ($1\le p<\infty$),  
   $S_\infty$ the space of compact operators on Hilbert space,
and   $S_1$   the trace class.
We will denote by  $S_\infty(R) $ (resp. $S_\infty(C) $)  the space  of all sequences $x=(x_n)$ with $x_n\in S_\infty$
such that   the series $\sum^\infty_1 x_nx^*_n$ (resp.\ $\sum^\infty_1 x^*_nx_n$) converges in
norm, and we equip it with the norm
 $ \|x\|_R = \left\|\left(\sum x_jx^*_j\right)^{1/2}\right\|$
 (resp. $ \|x\|_C = \left\|\left(\sum x^*_jx_j\right)^{1/2}\right\|$).
We then define
$$A_0=S_\infty(R) \cap S_\infty(C) $$
 and we  equip it   with the norm
$\|x\|_{A_0} = \max\{\|x\|_R, \|x\|_C\}$.

 We denote by  $S_1(R) $ (resp. $S_1(C) $)  
   the space   that was already introduced  as $M_*(R)$ (resp. $M_*(C)$) when $M_*=S_1,M=B(\ell_2)$, see \eqref{eqs1}.
 
   We then define
$$A_1=S_1(R) + S_1(C) 
$$
 and we  equip it   with the norm
$\|x\|_{A_1}  = \|x\|_{S_1(R)+S_1(C)}  
= \inf_{x=y+z} \|y\|_{S_1(R)} + \|z\|_{S_1(C)}.
$

Clearly, the couple $(A_0,A_1)$ forms a compatible couple in the sense of interpolation.
We denote for any $1<p<\infty$, $1\le q\le \infty$
\[
|||x|||_{p,q} = \|x\|_{(A_0,A_1)_{\theta,q}} \quad \text{where}\quad \theta=\frac{1-\theta}\infty + \frac\theta1 =\frac1p.
\]

Note that $A_1=A_0^*$ isometrically (in the usual duality defined by $\langle x, y\rangle= \sum {\rm tr}(x_n {}^t y_n)$) and 
  by a well known result (see \cite{CS}  and references there) this implies 
$(A_0,A_1)_{1/2,2}=\ell_2(S_2)$ isomorphically. Moreover,
the couple $(A^*_0,A^*_1)$ can be  clearly viewed as  compatible
and we have (see \cite[p. 54]{BL}) for $0 < \theta<1$ and $1\le q<\infty$
\begin{equation}\label{eq222}
(A_0,A_1)_{\theta,q}^*={(A^*_0,A^*_1)_{\theta,q'}}
\end{equation}
with equivalent norms. 
\begin{thm}\label{thm0}
Let $(\vp_n)$ be the usual independent $\pm 1$ valued random variable (``Rademacher functions'').
   Then for any $1<p<\infty$, $1\le q< \infty$ 
and for any  finite sequence $ x=(x_n)_n$ of operators in  $ S_{p,q}$
\begin{align}\label{eq1}
 \quad \int \left\|\sum \vp_nx_n\right\|_{S_{p,q}} d\mu &\simeq |||x|||_{p,q}\\
\intertext{and also}
\label{eq1b} \left\|\sum \vp_n\otimes x_n\right\|_{L_{p,q}(\mu\times {{\rm tr}})}
&\simeq |||x|||_{p,q}
\end{align}
where $\mu$ is the usual probability on $\{-1,1\}^{\bb N}$. Moreover, \eqref{eq1b} remains valid for $q=\infty$.\\
Here $A\simeq B$ means there are positive constants
$c_{p,q}$ and $C_{p,q}$ such that $c_{p,q}A\le B\le C_{p,q}A$.
\end{thm}
\begin{rk} It is not difficult to extend this Theorem to the case when the trace on $B(\ell_2))$ is replaced by
any semifinite faithful normal trace on a von Neumann algebra, but we  choose for simplicity to present
the details only in the case of $B(\ell_2)$. Indeed, all the ingredients for this extension now
exist in the literature (see \cite{PX}).
\end{rk}
\begin{rem}\label{rem0} Note that the space $M(R)\cap M(C)$ with $M=B(\ell_2)$ considered 
in \S \ref{sec1} is nothing but the bidual $A_0^{**}$ of $A_0$. Using this (and truncation of matrices
in the most usual way), one can check that, for any $x\in A_0+A_1$ and any $t>0$, we have
$K_t(x;A_0^{**},A_1)=K_t(x;A_0,A_1)$ and hence the norms
of the spaces $(A_0,A_1)_{\theta,q}$ and $(A_0^{**},A_1)_{\theta,q}$
coincide on any such $x$ for any $0<\theta<1$ and $ 1\le q\le \infty$.

\end{rem}
\begin{rem}\label{rem1}
Note that by interpolation for any $1<p<\infty$ (and $1\le q\le \infty$) the orthogonal projection onto $\ovl{\rm span}[\vp_n]$ is bounded on $L_{p,q}({{\mu}}\times{{\rm tr}})$. Indeed, we know that it is bounded on $L_{p_0}({{\mu}} \times {{\rm tr}})$, $L_{p_1}({{\mu}}\times {{\rm tr}})$ for $1<p_0<p<p_1$, and then we can use interpolation, together with the classical reiteration theorem (see \cite[p. 48]{BL}).
\end{rem}

\begin{rem}\label{rem2}
The equivalence between $[\vp_n]$ and Sidon lacunary sequences such as $[z^{2^n}]$ proved in \cite{P00} implies that for any $1<p<\infty$ and $1\le q\le \infty$
\[
\left\|\sum \vp_n\otimes x_n\right\|_{L_{p,q}(\mu\times {{\rm tr}})} \simeq \left\|\sum z^{2^n}\otimes x_n\right\|_{L_{p,q}(m\times {{\rm tr}})}
\]
where $m$ is  normalized Haar measure on ${\bb T}$. Again this is true by \cite{P00}  on $L_{p_i}$, $p_0<p<p_1$ (with simultaneous complementation) so this follows by interpolation.
\end{rem}

\begin{rem}\label{rem3} Let $E$ be any Banach space.
When $1\le p<\infty$, we   denote by  $L_p(E)$ the space of $E$-valued functions $p$-integrable 
on the unit circle, 
in Bochner's sense,  with the usual norm. We denote by $H^p(E)$ the subspace
of all functions $f$ with Fourier transform $\hat f$  supported on the non-negative integers.
The case $p=\infty$ is slightly different. We denote
 $L_\infty(B(\ell_2))$  the space of
essentially bounded  weak-$*$ measurable functions on the unit circle with values in $B(\ell_2))$,
equipped with the sup-norm. We again denote by $H^\infty(B(\ell_2))$
(resp.   $\ovl H^\infty_0(B(\ell_2))$ ) the subspace
formed of all functions with Fourier transform supported on the non-negative (resp. negative) integers.

By \cite[Cor. 3.4]{P7} we know that the pair $(H^\infty(B(\ell_2)), H^1(S_1))$ is $K$-closed in $(L_\infty(B(\ell_2)), L_1(S_1))$. From this it follows that:
\begin{align*}
\left\|\sum\nolimits^\infty_1 z^{2^n}x_n\right\|_{(H^\infty (B(\ell_2)), H^1(S_1))_{\theta,q}} &\simeq \left\|\sum z^{2^n}x_n\right\|_{(L_\infty (B(\ell_2)), H^1(S_1))_{\theta,q}}\\
&\simeq \left\|\sum z^{2^n} \otimes x_n\right\|_{L_{p,q}(dm\times {{\rm tr}})}.
\end{align*}
\end{rem}

\begin{rem}\label{rem4}
From \cite{LPP} we know that the mapping      $T\colon \ H^1(S_1)\to  A_1$
defined by $Tf=(\hat f(2^n))_n$  is a
   bounded surjection from $H^1(S_1)$ to $A_1 = S_1(R) +S_1(C)$. 
   Moreover, the (adjoint) map taking $x=(x_n)$ to $\sum z^{2^n} x_n$ is bounded
   from $A_1 = S_1(R) +S_1(C)$ to $H^1(S_1)$.
  Using the identity $H^1(S_1)^*=L_\infty(B(\ell_2))/\ovl H^\infty_0(B(\ell_2))$,
   by duality,  this implies that  
\[
T \text{ is bounded from } L_\infty(B(\ell_2))/\ovl H^\infty_0(B(\ell_2)) = H_1(S_1)^* \text{ onto }  A^*_1.
\]
By interpolating, we find
\[
T\colon \ (L_\infty(B(\ell_2))/\ovl H^\infty_0(B(\ell_2)), H^1(S_1))_{\theta,q} \to (A_1^*,A_1)_{\theta,q}.
\]
Note that the natural  ''inclusion" $H^\infty(B(\ell_2))\to L_\infty/\ovl H^\infty_0(B(\ell_2))$  trivially has norm $\le 1$. Therefore:
\[
T\colon \ (H^\infty(B(\ell_2)), H^1(S_1))_{\theta,q} \to (A_1^*,A_1)_{\theta,q}.
\]
Note that $A_1^*=A_0^{**}$. Therefore, using Remark \ref{rem0}, it is easy to check that
any \emph{finite} sequence $x=(x_n)$ with $x_n\in S_{p,q}\subset S_\infty$, the norms
of $(A_1^*,A_1)_{\theta,q}$ and $(A_0,A_1)_{\theta,q}$ coincide on $x$
for any $1<p<\infty, 1\le q\le \infty$ (here $\theta=1/p$).\\
Thus, invoking again \cite[Cor. 3.4]{P7},  by the preceding two remarks,  we find
\end{rem}

\begin{lem}\label{lem5}
Let $1<p<\infty, 1\le q\le \infty$. Then
$T$ is bounded from the subspace of all analytic polynomials
in  $L_{p,q}(dm\times \text{\rm tr})$ into $(A_0,A_1)_{\theta,q}$.
In particular, for some $c$, for all analytic polynomial $f$ with coefficients in $S_{p,q}=L_{p,q}( \text{\rm tr})$
\[
|||(\hat f(2^n))_n|||_{p,q} \le c\|f\|_{L_{p,q}(dm\times \text{\rm tr})}.
\]
\end{lem}

\begin{proof}[First part of the proof of Theorem \ref{thm0}]
Taking $f=\sum z^{2^n} x_n$ in the preceding Lemma we get
\[
|||x|||_{p,q} \lesssim \left\|\sum z^{2^n}x_n\right\|_{L_{p,q}(dm\times {{\rm tr}})}.
\]
By duality we get (using Remark \ref{rem1} and \eqref{eq222})
\[
\left\|\sum x^{2^n}x_n\right\|_{L_{p',q'}(dm\times {{\rm tr}})} \lesssim |||x|||_{p',q'}
\]
and since this holds for \emph{all} $1<p<\infty, 1\le q\le\infty$ (the case   $q=\infty,q'=1$ requires a  minor adjustment, see \cite[Remark, p.55]{BL}
). We deduce
\begin{align*}
|||x|||_{p,q} &\simeq \left\|\sum z^{2^n}x_n\right\|_{L_{p,q}(dm\times {{\rm tr}})}\\
\intertext{and hence by Remark~\ref{rem2}}
&\simeq \left\|\sum \vp_nx_n\right\|_{L_{p,q}(d\mu\times {\rm tr})}.
\end{align*} This proves \eqref{eq1b}.
\end{proof}
 
To complete the proof we use the following rather standard

\begin{lem}\label{lem6} For any
$ 1<p<\infty, 1\le q< \infty$,  any $ f\in L_{p,q}(\mu \times \text{\rm tr})$ 
of the form $f = \sum \vp_n x_n$ with $x_n\in S_{p,q}=L_{p,q}( \text{\rm tr})$ and any $r<\infty$ we have         
\[
\|f\|_{L_r(S_{p,q})} \lesssim \|f\|_{L_{p,q}(\mu\times \text{\rm tr})} .
\]
\end{lem}
\begin{proof}
We will use repeatedly Kahane's inequality for which we refer to \cite{P00} or \cite[p. 100]{LT}.
By $K$-convexity, the mapping $P\colon \ f\in L_{p,p}(\mu\times {\rm tr}) \mapsto L_r(\mu;S_p)$ defined by (projection onto span$[\vp_n]$)
\[
Pf  = \sum \langle f,\vp_n\rangle\vp_n
\]
is bounded for any value of $1<p<\infty$ and (by Kahane's inequality) $1\le r<\infty$. By interpolation, it follows that $P$ is bounded from
\[
(L_{p_0,p_0}(\mu\times {\rm tr}), L_{p_1,p_1}(\mu\times {\rm tr}))_{\alpha,q} \to (L_r(S_{p_0}), L_r(S_{p_1}))_{\alpha,q}
\]
for any $0<\alpha<1$, $1\le q\le \infty$.

We may choose $p_0<p<p_1, \frac1p = \frac{1-\alpha}{p_0} + \frac\alpha{p_1}$ and $r\ge q$ (here we use $q<\infty$), we then get
\[
P\colon  \ L_{p,q}(\mu\times {\rm tr}) \to (L_r(S_{p_0}), L_r(S_1))_{\alpha,q}
\]
but now $r\ge q$ guarantees that $L_q(L_r) \subset L_r(L_q)$  and hence  
\begin{align*}
(L_r(S_{p_0}), L_r(S_{p_1}))_{\alpha,q} \subset &L_r((S_{p_0}, S_{p_1})_{\alpha,q})=
L_r(S_{p,q}).
\end{align*}
Thus Lemma \ref{lem6} follows by restricting to $f=\sum \vp_nx_n$.
\end{proof}

\begin{proof}[End of proof of Theorem \ref{thm0}]
By Lemma \ref{lem6} we get
\[
\left\| \sum \vp_nx_n\right\|_{L_r(\mu; S_{p,q})} \lesssim |||x|||_{p,q}
\]
and hence again by duality
\[
|||x|||_{p',q'} \lesssim \left\|\sum \vp_nx_n\right\|_{L_{r'}(\mu; S_{p',q'})}.
\]
Then we conclude since Kahane's inequality  allows us to replace both $r$ and $r'$ by, say, 2.
\end{proof}
\begin{rk} Using Fernique's inequality in place of Kahane's (see \cite{LT})
it is easy to see that the preceding proof of  \eqref{eq1} and \eqref{eq1b} extends
to the case when $(\vp_n)$ is replaced by a sequence of i.i.d. Gaussian normal random variables.
This implies \eqref{eq1} and \eqref{eq1b}.
Note however 
(see \cite[p. 253]{LT}) that the Gaussian and Rademacher averages
 are not equivalent when $q=\infty$.
\end{rk}

\begin{rem}\label{rem9} Let $A_\theta = (A_0,A_1)_\theta$. 
Let us denote by $\text{Rad}(S_p) $ the closure in $L_{p}(\mu \times {\rm tr})$ 
of the set of finite sums
$  \sum \vp_n x_n$ with $x_n\in S_p$. By the description 
of $A_\theta$ obtained in  
  \cite{P2} (see also \cite{DR}), we know that if we  define $p$ by $\frac1p = \frac{1-\theta}\infty + \frac\theta1$ by Theorem~\ref{thm0}
\[
(A_0,A_1)_{\theta,p} = \text{Rad}(S_p) = \begin{cases}
                                         A_\theta\cap A_{1-\theta}&\text{if ~~$\theta\le \dfrac12$}\\
\noalign{\smallskip}
A_\theta+A_{1-\theta}&\text{if ~~$\theta\ge \dfrac12$.}
                                         \end{cases}
\]
and in particular
\[
(A_0,A_1)_{\theta,p} \simeq (A_0,A_1)_\theta.
\]
But this can also be seen using  \eqref{eq60} and 
the identity  $L_{p,p}  =L_{p}  $ relative to $\varphi \times {\rm tr}$
(and  a simultaneous complementation argument) where  $({\cl M},\varphi)$ is  the free group $\text{II}_1$ factor.
\end{rem}
From Theorem~\ref{thm0}, it is natural to search for an equivalent of the $K$-functional for $(A_0,A_1)$:
\n {\bf Problem:} Find an explicit  description (presumably in terms of $x$, $R$ and $C$) of
\[
K_t(x; A_0,A_1).
\]
\section{Remarks on real interpolation and non-commutative Khintchine}

We need more notation about the Lorentz space version of  the Schatten classes.

\n {\bf Notation:} 
We denote by ${X^c_{p,q}}$ (resp. ${X^r_{p,q}}$) the space of all sequences $x=(x_n)$
with $x_n\in S_{p,q}$ such that the series $\sum x^*_nx_n$  (resp. $\sum x_nx ^*_n$)
 (assumed w.o.t.\  convergent) satisfies
 $\left(\sum x^*_nx_n\right)^{1/2}\in S_{p,q}$ (resp. $\left(\sum x_nx^*_n\right)^{1/2}\in S_{p,q}$).
 We equip these spaces with the norms:
\[
\|(x_n)\|_{X^c_{p,q}} \overset{\text{\rm def}}{=} \left\|\left(\sum x^*_nx_n\right)^{1/2}\right\|_{S_{p,q}}
\quad
{\rm and}\quad  \|(x_n)\|_{X^r_{p,q}} \overset{\text{\rm def}}{=} \|(x^*_n)\|_{X^c_{p,q}}.\]
Lastly, we set $$X^c_p=X^c_{p,p} \quad{\rm and}\quad X^r_p=X^r_{p,p}.$$

Recall the following facts and terminology from \cite{P7}. 
Consider a compatible couple $(X_0, X_1)$ of Banach
(or quasi-Banach) spaces. Assume given a closed subspace
${\cl S}\subset X_0 + X_1$ and let

$${\cl S}_0 = {\cl S}\cap X_0,\qquad {\cl S}_1  = {\cl S}\cap X_1.$$

\n Let $Q_0 = X_0/{\cl S}_0$ and $Q_1 = X_1/{\cl S}_1$ be the associated
quotient spaces. Clearly $(Q_0, Q_1)$ form a
compatible couple since there are natural inclusion
maps $$Q_0 \to (X_0 +X_1)/{\cl S} \quad {\rm and} \quad Q_1 \to
(X_0 + X_1)/{\cl S}.$$
 \n We   say that the couple $({\cl S}_0, {\cl S}_1)$ is
$K$-closed (relative to $(X_0, X_1)$) if there is a
constant $c$ such that
 \begin{equation}{\forall t>0 \quad \forall x\in {\cl S}_0 + {\cl S}_1\quad
K_t(x; {\cl S}_0, {\cl S}_1) \le c K_t(x; X_0, X_1).}
\end{equation}
\n  In the terminology of \cite{P7},  $(Q_0, Q_1)$ is called $J$-closed if, for some constant $c$,
any element  $x\in Q_0\cap Q_1$ admits a simultaneous lifting  $\hat x$ satisfying
  $$\forall t >0 \quad    
J_t(\hat x; X_0, X_1) \le  c J_t(x; Q_0, Q_1).$$
In the present paper we make the convention to say that $({\cl S}_0,{\cl S}_1)$ is $J$-closed
when the pair $(Q_0, Q_1)$ is   $J$-closed in the above sense.\\
Equivalently, we will say that $({\cl S}_0,{\cl S}_1)$ is $J$-closed if there is a constant   $c$ such that for any $x\in X_0\cap X_1$ there is 
a single $\widehat x\in S_0\cap S_1$ satisfying simultaneously $\text{dist}_{X_j}(x,\widehat x) \le c \text{ dist}_{X_j}(x,{\cl S}_j)$ both for $j=0$ and $j=1$. The basic (simple but useful) fact
on which \cite{P7} rests, is  that, with our new terminology,  $K$-closed  is equivalent to $ J$-closed (this statement should
not be confused with the more obvious fact associated to the duality between
the $K$ and  $J$ norms or between subspaces and quotients).  
Note also that this is valid for quasi-normed spaces. 

\n {\bf Example:} Consider $1\le p_0,p_1\le \infty$. Let $X^c_j = X^c_{p_j}$, $X^r_j = X^r_{p_j}$, for $j=0,1$.
Let $X_j = X^c_j \oplus X^r_j$ and ${\cl S}_j\subset X_j$, ${\cl S}_j \overset{\text{def}}{=} \{(x,- x)\}$.
 Then $${\cl S}_j \simeq X^c_j \cap X^r_j\quad{\rm and}\quad X_j/{\cl S}_j \simeq X^c_j + X^r_j.$$ 
 Note also the special case: if $p_1=2$, $X_1/{\cl S}_1\simeq {\cl S}_1 \simeq \ell_2(S_2)$.

\begin{pro}\label{pro11} The above pair $({\cl S}_0,{\cl S}_1)$ is $J$-closed (and hence $K$-closed) in the following cases: 
\item[\rm (i)]  If $0< p_0 < 2$ and $p_1=2$. 
\item[\rm (ii)] If $2 < p_0 \le \infty$ and $p_1=2$.\\
Consequently,  in both cases, for any $p$ strictly between $p_0$ and $p_1=2$ and any $1\le q\le \infty$,  with $\theta$ defined by  $1/p = (1-\theta)/{p_0} +\theta/2$,  we have (with equivalent norms)
\[
({\cl S}_0, {\cl S}_1)_{\theta,q} = X^c_{p,q}\cap X^r_{p,q}\quad{\rm and}\quad
(X_0/{\cl S}_0, X_1/{\cl S}_1)_{\theta,q}=X^c_{p,q}+ X^r_{p,q}.
\]

\end{pro}

\begin{proof}
To show $({\cl S}_0,{\cl S}_1)$ is $J$-closed is the same as showing the following\\
{\bf Claim}: $\exists c>0$ such that the following holds. Given $x=(x_n)$, $x_n\in S_{p_0}$ there are $y=(y_n)$ and $z=(z_n)$ such that $x=y+z$ and we have \emph{simultaneously} for $j=0$ and $j=1$
\[
\|y\|_{X^r_{p_j}} + \|z\|_{X^c_{p_j}} \le c\|x\|_{X^r_{p_j}+X^c_{p_j}}.
\]
To prove this, we fix $\vp>0$ and choose $y^0,z^0$ such that $x=y^0+z^0$ and $\|y^0\|_{X^r_{p_0}} + \|z^0\|_{X^c_{p_0}} < \|x\|_{X^r_{p_0} + X^c_{p_0}}+\vp$.

Let then $\xi = (\sum y^0(n)y^0(n)^*)^{1/2}$, $\eta = (\sum z^0(n)^* z^0(n))^{1/2}$. We have $\|\xi\|_{p_0} = \|y^0\|_{X^r_{p_0}}$, $\|\eta\|_{p_0} = \|z^0\|_{X^c_{p_0}}$. We can assume (by perturbation) that $\xi>0$ and $\eta>0$ so that $\xi,\eta$ are invertible. We introduce the states
\[
f = \xi^{p_0} ({\rm tr}(\xi^{p_0}))^{-1},\quad g = \eta^{p_0}({\rm tr} (\eta^{p_0}))^{-1}.
\]
Let $a$ be defined by $\frac1a = \frac1{p_0} - \frac12$. We use the decomposition:
\[
x = (\|\xi\|_{p_0} f^{\frac1a})\widehat y^0 + \widehat z^0\cdot (\|\xi\|_{p_0} g^{\frac1a})
\]
where $\widehat y^0 = (\|\xi\|_{p_0} f^{\frac1a})^{-1}y^0$ and $\widehat z^0 =  z^0\cdot(\|\xi\|_{p_0} g^{\frac1a})^{-1}$. We note that $f = (\xi\|\xi\|^{-1}_{p_0})^{p_0}$ we have
\begin{align*}
\sum\|\widehat y^0(n)\|^2_2 &= {\rm tr}  \sum \widehat y^0(n){\widehat{y}^0(n)}^{*} = {\rm tr}(\|\xi\|^{-2}_{p_0} f^{-\frac1a} \sum y^0(n)y^0(n)^* f^{-\frac1a})\\
&= {\rm tr}(f^{-\frac1a}(\xi/\|\xi\|_{p_0})^2 f^{-\frac1a}) = {{\rm tr}}(f^{\frac2{p_0}-\frac2a})\\
&= {{\rm tr}}(f)=1.
\end{align*}
Similarly $\sum\|\widehat z^0(n)\|^2_2=1$.

To simplify notation let $\alpha=\|\xi\|_{p_0} f^{\frac1a}$ and $\beta = \|\xi\|_{p_0} g^{\frac1a}$. We have $x = \alpha\cdot \widehat y^0 + \widehat z^0\cdot \beta$. Let $T$ be the map (Schur multiplier with respect to the bases in which $\alpha,\beta$ are diagonal) defined by $T(t)  = \alpha t+t\beta$, $t\in S_\infty$. Consider the decomposition
\begin{align*}
x &= T(T^{-1}(x))\\
&= \alpha T^{-1}(x) + T^{-1}(x)\beta.
\end{align*}
We set $y_n = \alpha T^{-1}(x_n)$ and $z_n = T^{-1}(x_n)\beta$. Clearly (since $0 \le \frac{\alpha_i}{\alpha_i+\beta_j} \le 1$) we have
\[
\|y_n\|_{S_2} \le \|x_n\|_{S_2} ,\quad \|z_n\|_{S_2} \le \|x_n\|_{S_2} 
\]
and hence $\|y\|_{\ell_2(S_2)} \le \|x\|_{\ell_2(S_2)}, \|z\|_{\ell_2(S_2)} \le \|x\|_{\ell_2(S_2)}$. But also: $x_n = \alpha \widehat y^0(n) + \widehat z^0(n)\beta$
\begin{align*}
\|T^{-1}(x_n)\|_{S_2} &= \left\|\frac1{\alpha_i+\beta_j} [\alpha_i\widehat y^0(n)_{ij} + z^0(n)_{ij}\beta_j]\right\|_{S_2}\\
&\le \|\widehat y^0(n)\|_{S^2} + \|\widehat z^0(n)\|_{S_2}\\
\Rightarrow\quad \left(\sum\|T^{-1}(x_n)\|^2_{S_2}\right)^{1/2} &\le \left(\sum\|\widehat y^0(n)\|^2_{S_2}\right)^{1/2} + \left(\sum\|\widehat z^0(n)\|^2_{S_2}\right)^{1/2}\\
&\le 2.
\end{align*}
Therefore we conclude that (we set $t_n = T^{-1}(x_n)$)
\begin{align*}
\|y\|_{X^r_{p_0}} &= \|(\alpha T^{-1}(x_n))_n\|_{X^r_{p_0}} = \left\|\alpha\left(\sum t_nt^*_n\right)^{1/2}\right\|_{S_{p_0}}\\
&\le 2\|\alpha\|_{S_a} \le 2\|\xi\|_{p_0} = 2\|y_0\|_{X^r_{p_0}}
\end{align*}
and similarly
\[
\|z\|_{X^c_{p_0}} \le 2\|z_0\|_{X^c_{p_0}}.
\]
This proves our claim, whence (i) and (ii) follows by duality.
\end{proof}

\begin{cor}\label{cor1+}
Recall we set $A_0 = S_\infty(R) \cap S_\infty(C)$, $A_1 = S_1(R) + S_1(C)$. Then:
\begin{align*}
(A_0,A_1)_{\theta,q} &= X^r_{p,q} \cap X^c_{p,q}\quad \text{if}\quad  \theta<\frac12\\
(A_0,A_1)_{\theta,q} &= X^r_{p,q} + X^c_{p,q} \quad \text{if}\quad \theta>\frac12.
\end{align*}
\end{cor}
\begin{proof} As already mentioned,
by a well known self-duality result (note $A_1 \simeq A^*_0$) we have
\begin{align*}
(A_0,A_1)_{1/2,2} \simeq \ell_2(S_2),
\simeq X^r_{2,2}\cap X^c_{2,2}
\simeq X^r_{2,2} + X^c_{2,2}.
\end{align*}
So the corollary follows immediately by reiteration (see \cite[p. 48]{BL}).
\end{proof}
\begin{rk} The complex analogue of the preceding statement was proved in \cite[p. 109]{P000} (see also \cite{DR}).
\end{rk}
\begin{rk} It is known (see e.g. \cite{HK}  or \cite{JP})
 that the maps of the form $t\to \alpha T^{-1}(t)$, $t\to T^{-1}(t)\beta$ 
(corresponding to the Schur multipliers $[\alpha_i(\alpha_i+\beta_j)^{-1}]$ and $[\beta_j(\alpha_i+\beta_j)^{-1}]$) are c.b.\ on $S_Q$ for any $1<Q<\infty$. In particular,  in the preceding proof, we have for some constant $c_Q$
\begin{align*}
\|y\|_{X^r_Q} &\le c_Q\|x\|_{X^r_Q}\\
\|z\|_{X^c_Q} &\le c_Q\|x\|_{X^c_Q}
\end{align*}
and hence $\|y\|_{X^r_Q} + \|z\|_{X^c_Q} \le c_Q\|x\|_{X^r_Q\cap X^c_Q}$. 
Now let us fix $2 < Q<\infty$. Using this
simultaneous selection for the pair $(1,Q)$, one can show that   we have for any $1<p<Q$ and any $1\le q\le \infty$
\[
\|x\|_{X^r_{p,q}+X^c_{p,q}} \lesssim \|x\|_{(X^r_1+X^c_1, X^r_Q\cap X^c_Q)_{\theta,q}}
\]
and hence (reiteration again)
\[
\|x\|_{X^r_{p,q}+X^c_{p,q}} \lesssim \|x\|_{A_{p,q}}
\]
where
\begin{align*}
A_{p,q} \overset{\text{def}}{=} (A_0,A_1)_{\theta,q}\qquad A_0 = X^r_1+X^c_1\qquad
A_1=X^r_\infty \cap X^c_\infty.
\end{align*}
By duality, we also have
\[
\|x\|_{A_{p,q}} \lesssim \|x\|_{X^r_{p,q}\cap X^c_{p,q}}.
\]
\end{rk}

\begin{rk}
In particular we recover a result claimed in \cite[remark after Corollary 4.3]{LMS} (see also \cite{DR})
\[
\|x\|_{X^r_{2,q}+X^c_{2,q}} \lesssim \|x\|_{A_{2,q}} \lesssim \|x\|_{X^r_{2,q}\cap X^c_{2,q}}.
\]
\end{rk}

Exchanging the r\^oles of 1 and $\infty$, we obtain

\begin{cor}\label{cor2}
Let $B_0 = X^r_1\cap X^c_1$, $B_1 = X^r_\infty+X^c_\infty$. Then
\[
(B_0,B_1)_{\theta,q} = \begin{cases}
                       X^r_{p,q}\cap X^c_{p,q}&\text{if ~~$\theta<\dfrac12$}\\
\noalign{\smallskip}
X^r_{p,q}+X^c_{p,q}&\text{if ~~$\theta>\dfrac12$.}
                       \end{cases}
\]
\end{cor}

\begin{rk} It is not difficult to extend the results of this section
to a general non-commutative $L_p$-space (associated to a semi-finite
faithful normal trace) in place of $S_p$.
\end{rk}
\begin{rk} I do not know whether Proposition \ref{pro11} is valid for arbitrary $p_0\not = p_1$.
\end{rk}
\section{Connection with free probability}\label{sec6}

We refer to \cite{VDN} for all undefined notions in this section. Let $({\cl M},\varphi)$ be a ``non-commutative probability space'', i.e.\ $({\cl M},\varphi)$ is as $(M,\tau)$ was before but we impose $\varphi(1)=1$.

Let $(\xi_n)$ be a free family in ${\cl M}$. In what follows we assume that $(\xi_n)$ is either a free semi-circular (sometimes called ``free Gaussian'') family, or a free circular one (this corresponds to complex valued Gaussians) or a (free) family of Haar unitaries. The latter are the free analogues of Steinhaus random variables. They can be realized as the free generators of the free group ${\bb F}_\infty$ in the associated von~Neumann algebra (the so-called ``free group factor''). We could also include the free analogue of the Rademacher functions, i.e.\ a free family of copies of a single random choice of sign $\vp=\pm 1$. See \cite{RX} for a discussion of more general free families. Consider now a finite sequence $x=(x_n)$ with $x_n\in L_{p,q}(M,\tau)$.

The free analogue of Khintchine's inequality is the following fact that was (essentially) observed in \cite{HP2}. There are absolute positive constants $c,C$ such that for any $1\le p,q\le \infty$
\begin{equation}\label{eq60}
c|||x|||_{p,q} \le \left\|\sum \xi_n\otimes x_n\right\|_{L_{p,q}(\varphi\otimes\tau)}\le C|||x|||_{p,q}.
\end{equation}
In \cite{HP2} only the cases $p=1$ and $p=\infty$ are considered, but it is also pointed out there that the orthogonal projection onto $\overline{\text{span}}[\xi_n]$ is completely bounded on $L_p(\varphi)$ both for $p=1$ and $p=\infty$. From that simultaneous complementation, it is then routine to deduce \eqref{eq60}.
\begin{rk}
In particular, we have $|||x|||_{2,q} \simeq \|\sum \xi_n\otimes x_n\|_{L_{2,q}(\varphi\times\tau)}$. Perhaps our problem to compute more explicitly $|||x|||_{2,q}$ can be tackled by a more detailed study of the distribution of the ``non-commutative variable'' $\sum \xi_n\otimes x_n$. But while,  in Voiculescu's theory of the free Gaussian case, the ``$R$-transform'' (free analogue of the Fourier transform) can be calculated, it is not clear (even in case $\xi_n$,  $x_n$ and hence $\sum\xi_n\otimes x_n$ are     all       self-adjoint) how to use it to estimate  the spectral distribution of $\sum\xi_n\otimes x_n$ or, say, its $L_{2,q}$ norm. 
\end{rk}
Combining this with Theorem~\ref{thm0} (with a general $M$ in place of $B(H)$) we find:
\begin{prop}
For any $1\le p<\infty$ and $1\le q\le  \infty$ and for any sequence $x=(x_n)$ of operators in $L_{p,q}(\tau)$ we have
\begin{equation}\label{eq61}
\left\|\sum \vp_n\otimes x_n\right\|_{L_{p,q}(\mu\times\tau)} \simeq \left\|\sum \xi_n\otimes x_n\right\|_{L_{p,q}(\varphi\times\tau)}.
\end{equation}
\end{prop}

\n Note that this clearly fails for $p=\infty$ since $\text{span}_{L_\infty(\mu)}\{\vp_n\}\simeq \ell_1$ while $\text{span}_{L_\infty(\varphi)}\{\xi_n\}\simeq \ell_2$.

We refer the reader to \cite{HT2} for far reaching extensions of \eqref{eq60} or \eqref{eq61} involving random matrices. It is tempting to look for a direct, more conceptual proof of \eqref{eq61} but this has always eluded us (see however \cite{DR}). Note also that analogues of \eqref{eq60} and \eqref{eq61} (as well as \eqref{eq1}) are entirely open for $0<p<1$.

\begin{rem}\label{rk-fr}
Proposition \ref{pro11} (ii) yields some  information on the distribution function
of sums such as $S = \sum \xi_n \otimes x_n$.
Since the spans of the  variables $(\xi_n )$
are  completely complemented
simultaneously in $L_1(\varphi)$ and $L_\infty(\varphi)$ (a fortiori they form a $K$-closed pair),
it is easy to check that, if we set $S = \sum c_n \otimes x_n$
 or $S = \sum \lambda(g_n) \otimes x_n$, 
 uniformly over $t$
$$
K_t(x;A_1,A_0) \simeq K_t(S; L_1(\tau\times {\rm tr}), L_\infty(\tau\times {\rm tr}))=
    \int^t_0 S^\dagger(s) ds,$$
where $S^\dagger(s)$ denotes the generalized $s$-number
of $S$ in the sense of \cite{FK}.
Similarly,  by a well known fact (see \cite[p. 109]{BL}), we have uniformly over $t$
\begin{align*}
K_t(x;A_{1/2,2},A_0) \simeq K_t(S; L_2(\tau\times {\rm tr}), L_\infty(\tau\times {\rm tr}))
\simeq  \left(\int^{t^2}_0 S^\dagger(s)^2 ds\right)^{1/2}.
\end{align*}
The fact that
  the pair $(X^r_\infty\cap X^c_\infty, X^r_2\cap X^c_2)$ is $K$-closed,
  seems to yield some further information. Indeed, the $K_t$-functional for that pair can be estimated simply from the corresponding result for the (easy) pairs $(X^r_\infty,X^r_\infty)$ and $(X^c_\infty,X^c_2)$. So we have
\begin{equation}\label{eq22}
K_t(x;X^r_2\cap X^c_2, X^r_\infty \cap X^c_\infty)  \simeq \max\{K_t(x;X^r_2,X^r_\infty), K_t(x;X^c_2,X^c_\infty)\}  ,\end{equation}
and by a well known fact (see \cite[p. 109]{BL})
$$   K_t(x;X^r_2,X^r_\infty)\simeq\left(\int^{t^2}_0 \left(\sum x_nx^*_n \right)^\dagger(s) \ ds\right)^{1/2} \ {\rm and} \ 
K_t(x;X^c_2,X^c_\infty)\}\simeq\left(\int^{t^2}_0\left(\sum x^*_nx_n  \right)^\dagger(s) \ ds\right)^{1/2}.
$$    Therefore, we find both \eqref{eq22} and 
\begin{equation}\label{eq23}
 \left(\int^{t^2}_0 S^\dagger(s)^2 ds\right)^{1/2}   \simeq \max\{K_t(x;X^r_2,X^r_\infty), K_t(x;X^c_2,X^c_\infty)\} ,
\end{equation}
 where, of course,  the equivalences are meant with constants independent of $t$.

However, a very short and direct proof  was recently given in \cite {DR} that there is a constant $c$ such that
$$\forall t>0\quad  S^\dagger(ct)\le   \left((\sum x_nx^*_n)^{1/2}\right)^\dagger(t) + \left((\sum x^*_nx_n)^{1/2}\right)^\dagger(t),  
$$
from which \eqref{eq23} (and hence \eqref{eq22}) follows easily.
\end{rem}

\end{document}